\documentclass[12pt]{article}
\usepackage[top=1in, bottom=1in, left=1.2in, right=1.2in]{geometry}
\usepackage{graphicx, amsfonts,amssymb,amsthm,amsbsy,latexsym,amscd,amsmath,dsfont,euscript,enumerate,manfnt, marvosym,verbatim, calc,mathrsfs,marvosym,textcomp, color,xcolor,tikz-cd,setspace,stmaryrd,upgreek,bm} 
\usepackage{float} 
\usepackage{pgf,tikz,amsmath}
\usepackage{tikzit}

\tikzstyle{Doty}=[fill=black, draw=black, shape=circle]
\tikzstyle{doty}=[fill=black, draw=black, shape=circle]

\usepackage{blkarray}
\usepackage{caption}
\usepackage{enumerate}
\usepackage[pagebackref,colorlinks=true,
linkcolor=blue,
urlcolor=red,colorlinks,
citecolor=green]{hyperref}
\usepackage{cleveref}

\usepackage{parskip}
\usepackage{setspace}
\makeatletter
\def\thm@space@setup{%
  \thm@preskip=10pt 
  \thm@postskip=6pt 
}
\makeatother

\newtheorem{theorem}{Theorem}

\newtheorem{lemma}[theorem]{Lemma}
\newtheorem{claim}{Claim}

\theoremstyle{definition}

\newtheorem{conjecture}{Conjecture}

\newtheorem*{acknowledgement}{Acknowledgement}

\title{On Ramsey goodness of $K_{2,n}$ versus cycles}
\author{Abisek Dewan\footnote{Department of Mathematics and Statistics, IISER Kolkata, India (ad22rs069@iiserkol.ac.in)}\and Sayan Gupta \footnote{National Institute of Science Education and Research,  Bhubaneswar, India} \footnote{ Homi Bhabha National Institute, Traning School Complex, Anushakti Nagar, Mumbai 400094, India (sayan.gupta@niser.ac.in)} \and Rajiv Mishra\footnote{Department of Mathematics and Statistics, IISER Kolkata, India (rm20rs017@iiserkol.ac.in)}}
\date{}

\begin{document}

\maketitle
\begin{abstract}

A graph $G$ is called $H$-good if $R(G,H)=(|G|-1)(\chi(H)-1)+\sigma(H)$, where $\sigma(H)$ denotes the size of the smallest color class in a $\chi(H)$-coloring of $H$. In Ramsey theory, it is an interesting problem to study whether a graph $G$ is $H$-good or not. In this article, we study the Ramsey goodness of the pair $(K_{2,n},C_m)$, which naturally lies between the classical star-cycle and book-cycle problems.  We prove that
\begin{equation*}
    R(K_{2,n},C_{\{m,m+1\}})=m+1.
\end{equation*}
for all $m\ge 2n+2$, and consequently establish that 
\begin{equation*}
    R(K_{2,n},C_{m})=m+1.
\end{equation*}
for all $m\ge 3n+4$. This proves that $C_m$ is $K_{2,n}$-good in this range and improves a particular case of a result on the Ramsey goodness by Pokrovskiy and Sudakov.  Further, we provide a construction of a graph that disproves the $C_{2m}$-goodness of $K_{2,n}$ for all even $m$ satisfying $n\geq m+2$.
\end{abstract}

\section{Introduction}
Let $G=(V,E)$ be a graph. We introduce some standard notations used throughout this article. For a vertex $v\in V(G)$, we denote the set of neighbors of $v$ in $G$ by $N_{G}(v)$. The minimum and maximum degrees of $G$ are represented by $\delta(G)$ and $\Delta(G)$, respectively. Let $X\subseteq V(G)$. We denote the subgraph induced by $X$ as $G[X]$. The length of the largest cycle in a graph $G$ is called the \emph{circumference} and the length of the smallest cycle is called the \emph{girth}. They are denoted by $c(G)$ and $g(G)$, respectively. A graph is called \emph{$k$-connected} if the deletion of any $(k-1)$ vertices does not make the graph disconnected. The minimum such $k$ is defined as the \emph{connectivity} of the graph, and it is denoted by $\kappa(G)$. Let $C$ be a cycle on the vertices $\{1,2,3,\ldots ,n\}$. For vertices $i,j\in V(C)$, we denote the arc (or stretch) from $i$ to $j$ by
\[
[i,j]:=\{i,i+1,\ldots,j\},
\]
where addition is taken modulo $n$. The circular distance between $i$ and $j$ is the minimum of the lengths of the arcs $[i,j]$ and $[j,i]$.

Let $G$ and $H$ be two graphs. The \emph{Ramsey number}, denoted by $R(G,H)$, is the minimum positive integer $N$ such that any graph $\Gamma$ of order $N$ satisfies the property that either $\Gamma$ contains $G$ ($\Gamma \supseteq G$) or its complement contains $H$ ($\overline{\Gamma}\supseteq H$) as subgraphs. 
Let $G$ be a connected graph and $H$ be any graph satisfying $|V(G)|\geq \sigma(H)$ where $\sigma(H)$ is the size of the smallest color class in a $\chi(H)$-coloring of $H$. A foundational result from Burr \cite{burr1981ramsey} in 1981 provides a general lower bound on the Ramsey number of such pairs,
\begin{align*}
    R(G,H)\geq (|G|-1)(\chi(H)-1)+\sigma(H).
\end{align*}
We say $G$  is \emph{$H$-good} if equality holds in the above inequality. The notion of \emph{Ramsey goodness} was introduced by Burr and Erd\"{o}s \cite{burr1983generalizations}. Since then, there has been considerable study on the Ramsey goodness of different graph pairs. A \emph{star} is the bipartite graph $K_{1,n}$. A \emph{book} graph, denoted by $B_{n}$, is the graph consisting of $n$ triangles sharing a common edge called the \emph{base}. In graph Ramsey theory, star-cycle and book-cycle are two important and well-known graph pairs whose exact Ramsey value has been studied throughout the years. In 1973, Lawrence \cite{SLLawrence1973} obtained two exact values of the Ramsey number $R(K_{1,n},C_m)$.

\begin{theorem}[Lawrence \cite{SLLawrence1973}] \label{thm: Lawrence}
    \begin{align*}
R(K_{1,n},C_{m})&=\left\{\begin{array}{lll}
		2n+1& \textnormal{ if  }&  m \textnormal{ is odd and } m\leq 2n-1\\
        m &\textnormal{  if }& m\geq 2n.
    \end{array}\right.  
\end{align*}
\end{theorem}

\noindent Note that the first result in the above theorem implies that $K_{1,n}$ is $C_m$-good if $m$ is odd and $m\leq 2n-1$. The later one establishes the $K_{1,n}$-goodness of $C_{m}$ if $m\geq 2n$. Zhang et al.\cite{zhang2016narrowing} and Allen et al.\cite{allen2023ramsey} respectively extended the study of star-cycle Ramsey number for the remaining parametric conditions. However, these results do not explicitly study the Ramsey goodness of the pair $(K_{1,n},C_{m})$.

A similar study has also been done on the Ramsey goodness of the pair $(B_{n},C_{m})$. Faudree et al. \cite{faudree1991cycle} established that $B_{n}$ is $C_{m}$-good for odd $m\geq 5$ satisfying $n\geq4m-13$ and $C_{m}$ is $B_{n}$-good for all $m\geq 2n+2$. The latter one was further improved by Shi \cite{shi2010ramsey} for $m> 3n/2+7/4$.

\begin{theorem}[Faudree \emph{et al.} \cite{faudree1991cycle}] \label{thm: Faudree}
    \begin{align*}
R(B_{n},C_{m})&=\left\{\begin{array}{lll}
		2n+3& \textnormal{ if  }&  m\geq 5 \textnormal{ is odd and } n\geq 4m-13\\
        2m-1 &\textnormal{  if }& m\geq 2n+2.
    \end{array}\right.  
\end{align*}
\end{theorem}

The motivating observation is that by the deletion of the base edge from $B_n$, we get $K_{2,n}$. In particular,  we have
\begin{equation*}
    K_{1,n}\subset K_{2,n}\subset B_{n}.
\end{equation*}
 Since $K_{2,n}$ lies between a star and a book as a subgraph, it is natural to investigate whether the notion of goodness can be extended to the pair $(K_{2,n}, C_m)$. In \cite{gupta2025study}, Gupta settled the case for odd $m$ by showing that $K_{2,n}$ is $C_{m}$-good for odd $m$ under certain conditions.

\begin{theorem}[Gupta \cite{gupta2025study}]
If $n\geq 3493$, then for each odd $m\geq 7$ and $n\geq2m+499$,  
\begin{equation*}
 R(K_{2,n},C_{m})=2n+3.    
\end{equation*}
\end{theorem}

In this article, we have explicitly studied the Ramsey goodness of the pair $(K_{2,n},C_{m})$ for the remaining cases of $n$ and $m$. Our first main result deals with the construction of a graph which disproves the $C_{m}$-goodness of $K_{2,n}$ for even $m$. In particular, we show the following.

\begin{theorem}\label{thm:RamseyBadness in n>m+1}
If $n\geq m+2$ or $n\in \{m-2,m-1\}$,
$$R(K_{2,n},C_{2m})>n+m+1.$$
\end{theorem}

For $n\in\{m,m+1\}$, the goodness remains unknown. Now, it is an interesting question to ask whether $C_{m}$ is $K_{2,n}$-good or not. In this regard, there is a conjecture by Allen, Brightwell, and Skokan \cite{allen2013ramsey}.

\begin{conjecture}\label{conjecture}
   For any graph $H$ and $m\geq \chi(H)|H|$, $C_{m}$ is $H$-good.
\end{conjecture}
We prove this conjecture for $H=K_{2,n}$ and for all $m\geq 3n+4$. In order to prove this result, we first show the following.

\begin{theorem}\label{main theorem_2}
    If $m\geq 2n+2$ and $m\geq 3$, then $R(K_{2,n},C_{\{m,m+1\}})=m+1$.
\end{theorem}

\Cref{main theorem_2} proves that if $m\geq 2n+2$, then for $K_{2,n}$-free graph $G$ on $m+1$ vertices, $\overline{G}$ contains either a $C_{m}$ or $C_{m+1}$. If one can show the existence of a $C_{m}$ inside $\overline{G}$, then it ensures that $C_{m}$ is $K_{2,n}$-good for all $m\geq 2n+2$. However, we believe that there may be some counterexamples that contain a $C_{m+1}$ but not a $C_{m}$. We finally prove the $K_{2,n}$-goodness of $C_{m}$ for $m\geq 3n+4$.

\begin{theorem}\label{main result_3}
    If $m\geq 3n+4$, $R(K_{2,n},C_{m})=m+1$.
\end{theorem}

Pokrovskiy and Sudakov\cite{pokrovskiy2020ramsey} solved the \Cref{conjecture} for a general $k$-partite graph $H$ and $m$ sufficiently large integer. The result is as follows.

\begin{theorem}[Pokrovskiy et al. \cite{pokrovskiy2020ramsey}]\label{cycle goodness}
For $m \geq 10^{60} m_k$ and $m_k \geq m_{k-1} \geq \cdots \geq m_1$ satisfying $m_i \geq i^{22}$, we have
\[
R(C_m, K_{m_1,\ldots,m_k}) = (m - 1)(k - 1) + m_1.
\]
\end{theorem}

In our case, $k=2, m_{1}=2,m_{2}=n$. An improvement for complete bipartite graph in \Cref{cycle goodness} (Corollary~3.8 in \cite{pokrovskiy2020ramsey}) shows that $C_{m}$ is $K_{2,n}$-good for all $m\geq 2\times10^{49}n$ and $n\geq 8$. \Cref{main result_3} improves the bound to $m\geq 3n+4$ for $K_{2,n}$.

\section{Preliminary Results}

In this section, we introduce some preliminary results that will be used to prove our main results. We start with the following two results by Dirac \cite{dirac} on the length of the largest cycle present in a graph $G$.

\begin{lemma}[Dirac \cite{dirac}]\label{Dirac}
    If $G$ is a $2$-connected graph which satisfies $$|N_G(u)|+|N_G(v)|\geq k$$ for all non adjacent $u,v\in V(G)$, then $G$ contains a cycle of length at least $k$.
\end{lemma}

\begin{lemma}[Dirac \rm\cite{dirac}]\label{Dirac_2}
Let $G$ be a simple graph on $n \geq 3$ vertices. If $\delta(G) \geq \frac{n}{2}$, then $G$ is Hamiltonian.
\end{lemma}

Nash and Williams obtained a sufficient condition of Hamiltonicity for $2$-connected graphs. Observe that the bound on the minimum degree in the following result improves \Cref{Dirac_2} due to an extra condition of $2$-connectedness. Here, $\alpha(G)$ denotes the independence number of the graph $G$.

\begin{lemma}[Nash-Williams \cite{nash1971Hamiltonian}]\label{Thm:Nash-william}
    Let $G$ be a $2$-connected graph of order $n$ with $\delta(G)\geq \max\{(n+2)/3,\alpha(G)\}$. Then $G$ is Hamiltonian.
\end{lemma}

 A graph is called \emph{weakly pancyclic} if it contains cycles of all possible lengths between its girth and circumference. Weakly pancyclic graphs are often useful in the study of Ramsey numbers involving cycles. The following is a sufficient condition for a $2$-connected graphs to be weakly pancyclic.
 
\begin{lemma}[Brandt et al. \cite{brandt1998weakly}]\label{Thm:Weakpan-Brandt_nby4_250_1}
    Let $G$ be a $2$-connected non-bipartite graph of order $n$ with minimum degree $\delta(G) \geq  n/4 + 250$. Then $G$ is weakly pancyclic unless $G$ has odd girth $7$, in which case it has every cycle from $4$ up to its circumference except the $5$-cycle.
\end{lemma}

The following is one of the key results applied in one of the subcases arising in the proof of \Cref{Hamiltonian lemma}. For its proof, we refer the reader to Lemma~3.1 in \cite{dewan2026ramsey}.

\begin{lemma}[Dewan et al. \cite{dewan2026ramsey}]\label{lem:cycle_lemma}
Let $G$ be a $2$-connected graph on $n$ vertices such that
\[
|(N_G(u)\cup N_G(v))\setminus\{u,v\}|\geq k
\]
for all $u,v\in V(G)$. If $c(G)\geq n-k$, then
\[
c(G)\geq \min\{2k-2, n-1\}.
\]
\end{lemma}
It is an easy exercise to show that a graph $G$ is $K_{2,n}$-free if and only if for each $u,v\in V(G)$, $$|N_G(u)\cap N_G(v)|\leq n-1.$$
If a graph $G$ on $m$ vertices does not contain a $K_{2,n}$, then using the previous property we conclude that for each $u,v\in V(\overline{G})$, $$|(N_{\overline{G}}(u)\cup N_{\overline{G}}(v))\setminus \{u,v\}|\geq (m-2)-(n-1)=m-n-1.$$ We use this property frequently in our proofs. The following is another observation on $K_{2,n}$-free graphs which is used in the proof of \Cref{main theorem_2}
.
\begin{lemma}\label{2-connected lemma}
    Let $m\geq 2n+2$. If $G$ is a graph on $m+1$ vertices such that $ G\nsupseteq K_{2,n}$ and  $\overline{G}\nsupseteq C_{m}$, then $\overline{G}$ is $2$-connected. 
\end{lemma}
\begin{proof}
     If possible, let $\kappa(\overline{G})=0$. Then $\overline{G}$ is disconnected. Let $A$ be the smallest component of $\overline{G}$. This implies $|V(\overline{G})\setminus A|\geq (m+1)/2>n$. If $|A|\geq 2$, then we have a $K_{2,n}$ in $G$ which is a contradiction. So let $|A|=1$ and $A=\{a\}$. Then $a$ is adjacent to $(V(G)\setminus \{a\})=B$ in $G$. Consequently, for each $b\in B$,
    \begin{equation*}
        |N_{\overline{G}}(b)|\geq \left(m-1\right)-(n-1)\geq \frac{m}{2}.
    \end{equation*}
    As a result, $\overline{G}[B]$ is Hamiltonian by \Cref{Dirac_2} and contains a $C_{m}$, which is again a contradiction.
    
    Thus, we let $\kappa(\overline{G})=1$. That is, there exists a vertex $v$ such that $\overline{G}\setminus \{v\}$ is disconnected. However, we obtain a similar contradiction by showing that $\overline{G}\setminus \{v\}$ is Hamiltonian, which concludes that $\overline{G}$ is $2$-connected.
\end{proof}

\section{Proof of the main result}

Before proceeding to the main result, we prove some crucial observations regarding the longest cycle $C$ in a graph $G$. It shows that if $x$ and $y$ are two vertices outside the cycle, then there are certain restrictions on the adjacency of these vertices on $C$. The result is as follows.

 \begin{lemma}\label{observations}
  Let $C= v_{1}v_{2}\ldots v_{l}v_{1}$ be the longest cycle in $G$ and $X$ be the set of the remaining vertices outside $C$ with $|X|\geq 2$. Then for any two $x,y\in X$, we have the following:
      \begin{enumerate}
         \item $x$ (or $y)$ cannot be adjacent to any two consecutive vertices in $C$.

         \item If $x$ is adjacent to $v_{i}$ and $v_{j}$, then $y$ cannot be adjacent to $v_{i+1}$ and $v_{j+1}$ in $G$ simultaneously. Also $y$ cannot be adjacent to $v_{i-1}$ and $v_{j-1}$ in $G$ simultaneously.

          \item If $x$ is adjacent to $v_{i}$ and $v_{j}$, then $y$ cannot be adjacent to $v_{i+1}$ and $v_{j+2}$ in $G$ simultaneously. Also $y$ cannot be adjacent to $v_{i-1}$ and $v_{j-2}$ in $G$ simultaneously.

      \end{enumerate}
  \end{lemma}

  \begin{proof}
  Let $C=v_{1}v_2\ldots v_lv_1$ be the largest cycle in $G$ and $X$ be the set of remaining vertices. Suppose $x$ is adjacent to two consecutive vertices $v_i,v_{i+1}$ of $C$. Then $xv_{i+1}v_{i+2}...v_{i-1}v_ix$ is a cycle longer than $C$. Similarly, $y$ cannot be adjacent to two consecutive vertices of $C$. This proves (1). To prove (2), we let $x$ to be adjacent to $v_{i}$ and $v_{j}$. If $y$ is adjacent to both $v_{i+1}$ and $v_{j+1}$ then $xv_iv_{i-1}...v_{j+2}v_{j+1}yv_{i+1}v_{i+1}...v_{j-1}v_jx$ is a cycle longer than $C$ which is a contradiction. If $y$ is adjacent to $v_{i-1}$ and $v_{j-1}$ then $xv_iv_{i+1}...v_{j-2}v_{j-1}yv_{i-1}v_{i-2}...v_{j+1}v_jx$ is a cycle longer than $C$ in $G$ which is again a contradiction. The proof of (3)  follows in a similar manner. \end{proof}
In order to prove our main results, we need to prove the following.
\begin{theorem}\label{Hamiltonian lemma}
    Let $G$ be a $2$-connected graph on $m+1$ vertices such that $ G\nsupseteq C_{m}$, and
    \begin{equation*}
        |(N_{G}(u)\cup N_{G}(v))\setminus\{u,v\}|\geq \frac{m}{2}
    \end{equation*}
    for all $u,v\in V(G)$, then $G$ is Hamiltonian. 
\end{theorem}

\begin{proof}
 Suppose $G$ is not Hamiltonian. Consider a longest cycle $C$ of length $l$ in $G$. Since $G$ is $2$-connected therefore by \Cref{Dirac}, we conclude that $l\geq m/2$. Now if $l\geq m/2+1$ then by \Cref{lem:cycle_lemma}, we have
\begin{equation}\label{eqn:lemmafirsteqon_l}
    l\geq m-2.
\end{equation}
Now, we claim the following. 
\begin{claim}
    $l\neq \lceil m/2\rceil$.
\end{claim}
\begin{proof}[Proof of the claim]
    Suppose, on the contrary, $l=\lceil{m}/{2}\rceil$. Consider $X=V(G)\setminus V(C)$. Then $|X|=\lfloor{m}/{2}\rfloor+1$. We first show that $G[X]$ cannot be a null graph. If not, then let $x,y$ be two different vertices in $X$. By the condition, we have 
    $$ |(N_{G}(x)\cup N_{G}(y))\cap C|\geq \frac{m}{2}.$$ 
    By the pigeonhole principle and without loss of generality, we can assume that $|N_G(x)\cap C|\geq \lceil{m}/{4}\rceil$. This implies that $x$ has at least two consecutive neighbors on $C$, contradicting that $C$ is the longest cycle. It is important to note that if $l\geq m/2+1$, then this contradiction may not be possible. However, in that case, we can apply \Cref{lem:cycle_lemma} to conclude that $c(G)\geq m-2$.

    Now, we show that there exists a longest path $P$ in $G[X]$ with both endpoints having a neighbor in $C$. Note that, for any longest path $P=P_{s}=x_{1}x_{2}\ldots x_{s}$ in $G[X]$, at least one endpoint must have a neighbor in $C$ as 
    \begin{equation*}
        |(N_{G}(u)\cup N_{G}(v))\setminus\{u,v\}|\geq \frac{m}{2}
    \end{equation*}
    for all $u,v\in V(G)$ and $|X\setminus \{x_{1},x_{s}\}|=\lfloor{m}/{2}\rfloor-1$. Hence, we assume that $P_s=x_1 x_2\ldots x_s$ is a longest path such that $x_1$ has a neighbor in $V(C)$. Assume $x_s$ does not have a neighbor in $V(C)$. Since $G$ is $2$-connected, $x_s$ has at least  one more neighbor other than $x_{s-1}$ on $P_s$, say $x_{s'}$. Note that $x_{s'+1}$ must have a neighbor in $V(C)$ otherwise 
\begin{equation*}
    |(N_{G}(x_{s})\cup N_{G}(x_{s'+1}))\setminus\{x_{s},x_{s'+1}\}|\leq|X|-2\leq  \bigg\lfloor\frac{m}{2}\bigg\rfloor-1
\end{equation*}
which is a contradiction. Thus, we get a longest path $x_1\ldots x_{s'}\, x_s\, x_{s-1}\ldots x_{s'+1}$ with both endpoints having a neighbor in $C$.

Now assume $P_s=x_1 x_2\ldots x_s$ is a longest path in ${G}[X]$ such that $x_1$ and $x_s$ each have at least one neighbor in $V(C)$. By our condition on $G$, we have
\begin{equation}
    |\{N_{G}(x_1) \cup N_{G}(x_s)\}\cap V(C)|\geq \frac{m}{2}-s+2.
\end{equation}
We can choose two non-empty sets $A\subseteq N_{G}(x_1) \cap V(C)$ and $B\subseteq N_{G}(x_s) \cap V(C)$ such that $A\cap B=\phi$ and $|A|+|B|\geq {m}/{2}-s+2$. 
 Consider the disjoint stretches $A_1, A_2 \dots, A_q$ of $V(C)$ such that between any two consecutive stretches, there must exist at least one member of $B$ and for all $j\in [q]$
\begin{enumerate}
    \item $A_j=[t_j,h_j]$ where $t_j,h_j\in A$,
    \item  $B\cap  A_j=\phi$, and
    \item $A\subseteq  \sqcup_{i}A_i$.
\end{enumerate}
 Similarly, consider stretches $B_j=[t'_j,h'_j]$ for $j\in [q]$, where $t'_j,h'_j\in B$ and $B\subseteq \sqcup_iB_i$. Let $a_j=|A_j|$ and $b_j=|B_j|$ for $j\in[q]$. Since $P_s$ is a path between $x_1$ and $x_s$, the circular distance between any $a\in A$ and $b\in B$ is at least $s+1$. In particular, the number of vertices on the arc $[t_{j+1},h'_{j-1}]$ is at least $s$. From (1) of \Cref{observations}, it follows that no two vertices of $A$ or $B$ can be consecutive on $C$ as $C$ is the longest cycle. Thus we get
\begin{align*}
    \underset{i=1}{\sum ^{q}} a_i\geq \underset{i=1}{\sum ^{q}} \left(2|A_i\cap A|-1\right)=2|A|-q\\
 \underset{i=1}{\sum ^{q}} b_i\geq \underset{i=1}{\sum ^{q}} \left(2|B_i\cap B|-1\right)=2|B|-q.\\
    \end{align*}
Since there are $q$ many disjoint parts for both $A$ and $B$, we have
\begin{eqnarray*}
    l &\geq&  \underset{i=1}{\sum ^{q}} (a_i+b_i)+ 2qs\\
    &\geq& 2(|A|+|B|)-2q+2qs\\
    &\geq& 2\left(\frac{m}{2}-s+2\right)+2q(s-1)\\
&\geq& m+2 \text{  }\;\;\;\; (\text{As }q\geq 1 \text{ and } s\geq 1)
\end{eqnarray*}
which is absurd. This establishes our claim.
\end{proof}

Thus from \Cref{eqn:lemmafirsteqon_l}, $l\in \{m-2,m-1\}$ as ${G}$ does not contain a $C_{m}$ or $C_{m+1}$. Let $V(G)\setminus V(C)=X$.
 \begin{claim}
        $G[X]$ is a null graph.
    \end{claim}
\begin{proof}[Proof of the claim]
Note that the size of $X$ is either two or three. First, we assume that $X=\{x,y\}$ and if possible, let $x$ be adjacent to $y$. Then for each $a\in N_G(x)\cap C$ and $b\in N_G(y)\cap C$, the circular distance between $a$ and $b$ on $C$ is at least $3$; otherwise, we can construct a larger cycle through $x,y$, which is a contradiction. Also note that each $a\in  N_G(x)\cap C$ and $b\in  N_G(y)\cap C$ cannot be consecutive on the cycle. Since $|N_G(x)\cup N_G(y)|\geq m/2$, to place all the vertices of this set on the cycle $C$ in a suitable way, we need a stretch of at least $m$ vertices, which is not possible as $l=m-1$.

Therefore, we let $X=\{x,y,z\}$ and $x$ is adjacent to $y$. If neither $x$ nor $y$ is adjacent to $z$, then $|N_G(x)\cup N_G(y)|\geq m/2$ and we have a similar contradiction as the previous case. If at least one of $x$ or $y$ (say $y$) is adjacent to $z$. Then consider the pair $\{x,z\}$. Here we have $|N_G(x)\cup N_G(z)|\geq m/2-1$. Since $xyz$ is a path on three vertices and $C$ is the longest cycle, for each $a\in N_G(x)\cap C$ and $b\in N_G(z)\cap C$, the circular distance between $a$ and $b$ on $C$ is at least $4$. Thus, we need a stretch of at least 
\begin{equation*}
    (2|N_G(x)\cap C|-1)+(2|N_G(z)\cap C|-1)+3\geq 2|N_G(x)\cup N_G(z)|+1\geq m-1
\end{equation*}
vertices on $C$ to adjust $N_G(x)\cup N_G(z)$. This is a contradiction as $l=m-2$.
\end{proof}

Now it is enough to show that $l=m-1$ many vertices are not sufficient to place all the neighbors of $x$ and $y$. If $l=m-2$, then we choose any two $x,y\in X$ and obtain a similar contradiction. So, without loss of generality, let $l=m-1$. Since $x$ and $y$ satisfy
    \begin{equation}
        |(N_{G}(x)\cup N_{G}(y))\cap C|\geq \frac{m}{2},
    \end{equation}
    we can construct $A=N_G(x) \cap C$, $B=N_G(y) \cap C$ such that $A\cap B= \phi$ and
    \begin{align*}
        |A|=\frac{m}{4}+k\\
        |B|\geq \frac{m}{4}-k
    \end{align*}
    for some $k\geq 0$. Note that $|B|$ has to be non-zero, otherwise $|A|\geq m/2$, which implies that there exist at least two consecutive neighbors of $x$ on $C$ forming a longer cycle.

    We prove that there exist $a\in A$ and $b\in B$ such that they are consecutive on $C$. If not, then for each $a\in A$ and $b\in B$, the circular distance between $a$ and $b$ is at least $2$. Also, no two elements in $A$ or $B$ can be consecutive as $C$ is a longest cycle. Thus, we need a stretch of at least
    \begin{equation*}
        \left[2\left(\frac{m}{4}+k\right)-1\right]+\left[2\left(\frac{m}{4}-k\right)-1\right]+2=m
    \end{equation*}
     vertices to accommodate all the members of $A\sqcup B$, which is not possible.

Now, without loss of generality, we can index the vertices of $C$ as $\{1,2,\ldots,m-1\}$ such that $(m-1)\in A$ and $1\in B$. We divide the rest of the proof into two cases.

\textbf{Case 1.} There exists no $j\in A$ and $i,k\in B$ such that $i<j<k$.

Let $s=\min\{a:a\in A\}$ and $t=\max\{b:b\in B\}$. Then clearly, we have $s-t\leq 2$; otherwise, we need a stretch of at least $m$ vertices on $C$ to adjust the members of $A\sqcup B$. If $s-t=2$, then 
\begin{align*}
    &A=\{s,s+2,\ldots,m-3, m-1\}\\
    &B=\{1,3,5,\ldots,t\}
\end{align*}
It is important to note that this is the only possible case here. If for any two consecutive vertices $i,j\in A$ or $i,j\in B$, $|i-j|\geq 3$, then the number of vertices on the required stretch would be at least
\begin{equation*}
    \left[2\left(\frac{m}{4}+k\right)-1\right]+\left[2\left(\frac{m}{4}-k\right)-1\right]+1+1=m.    
\end{equation*}
See \Cref{fig:twocycles} for a better understanding of the scenario. Here, we observe the following: 
\begin{align*}
    &2\nsim j \textnormal{ for all } j\in [s+1,m-1]\\
    &2\nsim i \textnormal{ for all } i\in \{b+1:b\in B\setminus \{1\}\}.
\end{align*}
If $2\sim (s+j)\in [s+1,m-1]$ for some odd $j$, then $2(s+j)(s+j+1)\ldots (m-1)x(s+j-1)(s+j-2)\ldots 3y12$ forms a cycle of length $m+1$. Further, if $2\sim (s+j)\in [s+1,m-1]$ for some even $j$, then $2(s+j)(s+j+1)\ldots (m-1)x(s+j-2)\ldots 3y12$ forms a cycle of length $m$. If $2\sim i=(b+1)$ for some $b\in B$, then $2(b+1)(b+2)\ldots 1y(b-1)(b-2)\ldots 2$ forms a cycle of length $m$. Clearly, $y\nsim (b+1)$ for all $b\in B$. Also, by (2) and (3) of \Cref{lem:cycle_lemma} $y\nsim j$ for all $[s+1,m-1]$.  Since $|A|=m/4+k$, $|B|\geq m/4-k$ and no two elements of $A$ are consecutive on $C$,
\begin{equation*}
    |[s+1,m-1]|+|\{b+1:b\in B\setminus\{1\}\}|\geq \frac{m}{2}+2k-1+\frac{m}{4}-k-1\geq \frac{3m}{4}+k-2.
\end{equation*}
Consequently,
\begin{equation*}
        |(N_{G}(2)\cup N_{G}(y))\setminus\{2,y\}|\leq \frac{m}{4}-k+1<\frac{m}{2},
\end{equation*}
which is a contradiction.

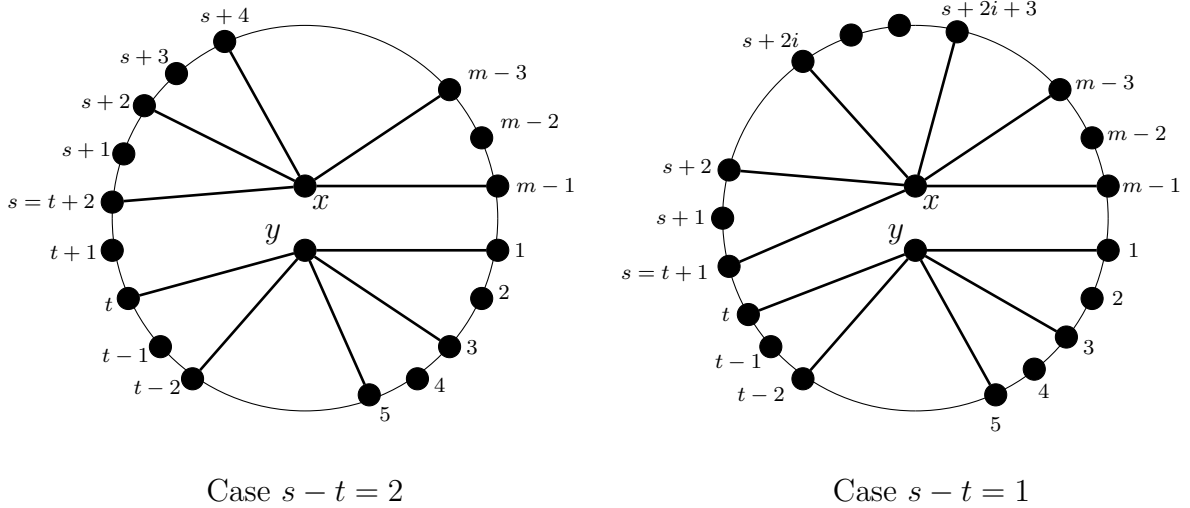
\begin{figure}[H]
    \centering
    \begin{tikzpicture}[x=0.85cm,y=0.85cm]
	\begin{pgfonlayer}{nodelayer}
		\node [style=none] (0) at (0, 3) {};
		\node [style=none] (1) at (3, 0) {};
		\node [style=none] (2) at (0, -3) {};
		\node [style=none] (3) at (-3, 0) {};
		\node [style=Doty,scale=0.75] (4) at (3, -0.5) {};
		\node [style=Doty,scale=0.75] (5) at (2.75, -1.25) {};
		\node [style=Doty,scale=0.75] (6) at (2.25, -2) {};
		\node [style=Doty,scale=0.75] (7) at (1.75, -2.5) {};
		\node [style=Doty,scale=0.75] (8) at (1, -2.75) {};
		\node [style=Doty,scale=0.75] (11) at (-2.75, -1.25) {};
		\node [style=Doty,scale=0.75] (12) at (-3, 0.25) {};
		\node [style=Doty,scale=0.75] (13) at (-2.82, 1) {};
		\node [style=Doty,scale=0.75] (14) at (-2.5, 1.75) {};
		\node [style=Doty,scale=0.75] (15) at (-2, 2.25) {};
		\node [style=Doty,scale=0.75] (16) at (-1.25, 2.75) {};
		\node [style=Doty,scale=0.75] (18) at (2.25, 2) {};
		\node [style=Doty,scale=0.75] (19) at (2.75, 1.25) {};
		\node [style=Doty,scale=0.75] (20) at (3, 0.5) {};
		\node [style=Doty,scale=0.75] (21) at (0, 0.5) {};
		\node [style=Doty,scale=0.75] (22) at (0, -0.5) {};
		\node [style=Doty,scale=0.75] (23) at (-2.25, -2) {};
		\node [style=Doty,scale=0.75] (24) at (-1.75, -2.5) {};
		\node [style=none] (25) at (3.35, -0.5) {\scriptsize$1$};
		\node [style=none] (26) at (3.1, -1.15) {\scriptsize$2$};
		\node [style=none] (27) at (2.6, -2.0) {\scriptsize$3$};
		\node [style=none] (28) at (2.1, -2.6) {\scriptsize$4$};
		\node [style=none] (29) at (1.25, -3.05) {\scriptsize$5$};
		\node [style=none] (30) at (-3.05, -1.35) {\scriptsize$t$};
		\node [style=none] (31) at (-2.79, -2.15) {\scriptsize$t-1$};
		\node [style=none] (32) at (-2.3, -2.65) {\scriptsize$t-2$};
		\node [style=none] (33) at (-3.95, 0.25) {\scriptsize$s=t+2$};
		\node [style=none] (34) at (-3.4, 1.05) {\scriptsize$s+1$};
		\node [style=none] (35) at (-3.1, 1.8) {\scriptsize$s+2$};
		\node [style=none] (36) at (-2.5, 2.5) {\scriptsize$s+3$};
		\node [style=none] (37) at (-1.25, 3.15) {\scriptsize$s+4$};
		\node [style=none] (38) at (3.75, 0.5) {\scriptsize$m-1$};
		\node [style=none] (39) at (3.5, 1.5) {\scriptsize$m-2$};
		\node [style=none] (40) at (3, 2.25) {\scriptsize$m-3$};
		\node [style=none] (41) at (0.25, 0.25) {$x$};
		\node [style=none] (42) at (-0.5, -0.25) {$y$};
		\node [style=none] (43) at (9.5, 3) {};
		\node [style=none] (44) at (12.5, 0) {};
		\node [style=none] (45) at (9.5, -3) {};
		\node [style=none] (46) at (6.5, 0) {};
		\node [style=Doty,scale=0.75] (47) at (12.5, -0.5) {};
		\node [style=Doty,scale=0.75] (48) at (12.25, -1.25) {};
		\node [style=Doty,scale=0.75] (49) at (11.85, -1.85) {};
		\node [style=Doty,scale=0.75] (50) at (11.35, -2.35) {};
		\node [style=Doty,scale=0.75] (51) at (10.75, -2.75) {};
		\node [style=Doty,scale=0.75] (52) at (6.9, -1.5) {};
		\node [style=Doty,scale=0.75] (53) at (6.6, -0.75) {};
		\node [style=Doty,scale=0.75] (54) at (6.5, 0) {};
		\node [style=Doty,scale=0.75] (55) at (6.6, 0.75) {};
		\node [style=Doty,scale=0.75] (57) at (7.75, 2.44) {};
		\node [style=Doty,scale=0.75] (58) at (11.75, 2) {};
		\node [style=Doty,scale=0.75] (59) at (12.25, 1.25) {};
		\node [style=Doty,scale=0.75] (60) at (12.5, 0.5) {};
		\node [style=Doty,scale=0.75] (61) at (9.5, 0.5) {};
		\node [style=Doty,scale=0.75] (62) at (9.5, -0.5) {};
		\node [style=Doty,scale=0.75] (63) at (7.25, -2) {};
		\node [style=Doty,scale=0.75] (64) at (7.75, -2.5) {};
		\node [style=none] (65) at (12.9, -0.5) {\scriptsize$1$};
		\node [style=none] (66) at (12.65, -1.25) {\scriptsize$2$};
		\node [style=none] (67) at (12.2, -2) {\scriptsize$3$};
		\node [style=none] (68) at (11.5, -2.7) {\scriptsize$4$};
		\node [style=none] (69) at (10.75, -3.2) {\scriptsize$5$};
		\node [style=none] (70) at (6.55, -1.5) {\scriptsize$t$};
		\node [style=none] (71) at (6.75, -2.25) {\scriptsize $t-1$};
		\node [style=none] (72) at (7.1, -2.75) {\scriptsize $t-2$};
		\node [style=none] (73) at (5.6, -0.85) {\scriptsize$s=t+1$};
		\node [style=none] (74) at (5.85, 0) {\scriptsize $s+1$};
		\node [style=none] (75) at (5.95, 0.8) {\scriptsize$s+2$};
		\node [style=none] (76) at (7.25, 2.75) {\scriptsize$s+2i$};
		\node [style=none] (78) at (13.18, 0.5) {\scriptsize$m-1$};
		\node [style=none] (79) at (12.95, 1.3) {\scriptsize$m-2$};
		\node [style=none] (80) at (12.45, 2.1) {\scriptsize$m-3$};
		\node [style=none] (81) at (9.75, 0.25) {$x$};
		\node [style=none] (82) at (9.2, -0.25) {$y$};
		\node [style=none] (83) at (-3.6, -0.55) {\scriptsize$t+1$};
		\node [style=Doty,scale=0.75] (84) at (-3, -0.5) {};
		\node [style=Doty,scale=0.75] (85) at (8.5, 2.85) {};
		\node [style=Doty,scale=0.75] (86) at (10.15, 2.9) {};
		\node [style=none] (87) at (10.65, 3.25) {\scriptsize$s+2i+3$};
		\node [style=Doty,scale=0.75] (88) at (9.25,3) {};
		\node [style=none] (90) at (0, -4.25) {Case $s-t=2$};
		\node [style=none] (91) at (9.75, -4.25) {Case $s-t=1$};
	\end{pgfonlayer}
	\begin{pgfonlayer}{edgelayer}
		\draw [bend left=45] (0.center) to (1.center);
		\draw [bend left=45] (1.center) to (2.center);
		\draw [bend left=45] (2.center) to (3.center);
		\draw [bend left=45] (3.center) to (0.center);
		\draw [line width=1pt](21) to (20);
		\draw [line width=1pt](21) to (18);
		\draw [line width=1pt](21) to (16);
		\draw [line width=1pt](21) to (14);
		\draw [line width=1pt](21) to (12);
		\draw [line width=1pt](22) to (11);
		\draw [line width=1pt](22) to (4);
		\draw [line width=1pt](22) to (6);
		\draw [line width=1pt](22) to (8);
		\draw [line width=1pt](22) to (24);
		\draw [bend left=45] (43.center) to (44.center);
		\draw [bend left=45] (44.center) to (45.center);
		\draw [bend left=45] (45.center) to (46.center);
		\draw [bend left=45] (46.center) to (43.center);
		\draw [line width=1pt](61) to (60);
		\draw [line width=1pt](61) to (58);
		\draw [line width=1pt](61) to (57);
		\draw [line width=1pt](61) to (55);
		\draw [line width=1pt](61) to (53);
		\draw [line width=1pt](62) to (52);
		\draw [line width=1pt](62) to (47);
		\draw [line width=1pt](62) to (49);
		\draw [line width=1pt](62) to (51);
		\draw [line width=1pt](62) to (64);
		\draw [line width=1pt](86) to (61);
	\end{pgfonlayer}
\end{tikzpicture}
    \caption{Distribution of neighbors of $x$ and $y$ on $C$}
    \label{fig:twocycles}
\end{figure}

 So we let $s-t=1$. If $A=\{s,s+2,\ldots, m-1\}$, then we proceed similarly to obtain a contradiction. Otherwise let 
 \begin{equation*}
     A=\{s,s+2,\ldots, s+2i\}\cup \{s+2i+3,s+2i+5,\ldots, m-1\}
 \end{equation*}
 for some fixed $i\geq 0$. No other case of $A$ is possible here (see \Cref{fig:twocycles}). Otherwise, we need at least $m$ vertices on $C$ for $A\sqcup B$. In this case, we have $2\nsim j$ for all $j\in [s,m-1]\setminus \{s+2i+3\}$. If not, then we can construct a similar cycle as the previous of length $(m+1)$ for $j\geq (s+1)$. For $j=s$, we get $2s(s+1)\ldots (m-1)1yt(t-1)\ldots 2$, a cycle of length $m$. Therefore, arguing similarly, we conclude that 
 \begin{equation*}
        |(N_{G}(2)\cup N_{G}(y))\setminus\{2,y\}|<\frac{m}{2},
\end{equation*}
a contradiction. 

\textbf{Case 2.} There exist $j\in A$ and $i,k\in B$ such that $i<j<k$.
For $q\geq 2$ and $j\in [q]$, let $B_j=[t_j, h_j]$ be stretches of $V(C)$ such that between any two consecutive stretches, there must exist at least one member of $A$ and
\begin{enumerate}
    \item $t_{1}=1$ and $t_{j}, h_{j}\in B$ for $j\in[q]$, 
    \item $A\cap (\sqcup_iB_i)=\phi$,
    \item $B\subseteq \sqcup_iB_i$.
\end{enumerate}

 That is, for any $h_{j}$ and $t_{j+1}$, there exists some $a\in A$ such that $h_{j}<a<t_{j+1}$. It is important to note that $B_{j} $ can be singleton also. In that case $t_{j}=h_{j}$. See a pictorial representation of these stretches in \Cref{fig:cycle_strech}.

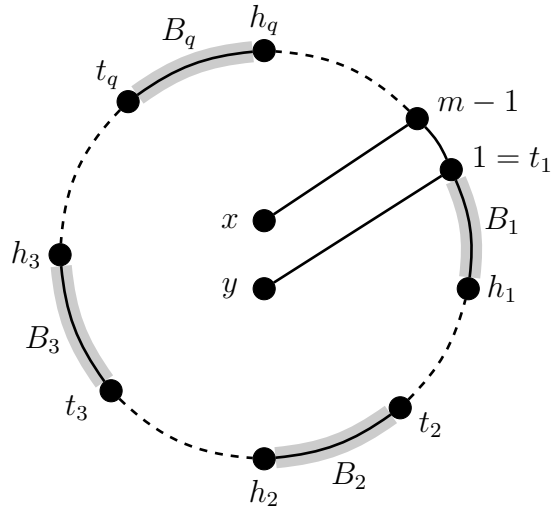
\begin{figure}[H]
    \centering
    \begin{tikzpicture}[x=0.9cm,y=0.9cm]
	\begin{pgfonlayer}{nodelayer}
		\node [style=Doty,scale=0.75] (0) at (0, 3) {};
		\node [style=Doty,scale=0.75] (2) at (0, -3) {};
		\node [style=Doty,scale=0.75] (3) at (-3, 0) {};
		\node [style=Doty,scale=0.75] (4) at (2.25, 2) {};
		\node [style=Doty,scale=0.75] (5) at (2.75, 1.25) {};
		\node [style=Doty,scale=0.75] (6) at (3, -0.5) {};
		\node [style=Doty,scale=0.75] (7) at (2, -2.25) {};
		\node [style=Doty,scale=0.75] (8) at (-2.25, -2) {};
		\node [style=Doty,scale=0.75] (9) at (-2, 2.25) {};
		\node [style=Doty,scale=0.75] (10) at (0, 0.5) {};
		\node [style=Doty,scale=0.75] (11) at (0, -0.5) {};
		\node [style=none] (12) at (-0.5, 0.5) {$x$};
		\node [style=none] (13) at (-0.5, -0.5) {$y$};
		\node [style=none] (14) at (3.15, 2.25) {$m-1$};
		\node [style=none] (15) at (3.65, 1.4) {$1=t_1$};
		\node [style=none] (16) at (3.5, -0.5) {$h_1$};
		\node [style=none] (17) at (2.45, -2.5) {$t_2$};
		\node [style=none] (18) at (0, -3.5) {$h_2$};
		\node [style=none] (19) at (-2.75, -2.25) {$t_3$};
		\node [style=none] (20) at (-3.5, 0) {$h_3$};
		\node [style=none] (21) at (-2.3, 2.65) {$t_q$};
		\node [style=none] (22) at (0, 3.45) {$h_q$};
		\node [style=none] (23) at (3.5, 0.5) {$B_1$};
		\node [style=none] (24) at (1.25, -3.25) {$B_2$};
		\node [style=none] (25) at (-3.25, -1.25) {$B_3$};
		\node [style=none] (26) at (-1.25, 3.25) {$B_q$};
	\end{pgfonlayer}
	\begin{pgfonlayer}{edgelayer}
		\draw [bend left=10,line width=1pt] (4) to (5);
		\draw [bend right=20,line width=1pt,dashed] (4) to (0);
         \draw [bend right=15,line width=1pt] (0) to (9);
		\draw [bend right=15,line width=8pt,opacity=0.2] (0) to (9);     
		\draw [bend right=20,line width=1pt,dashed] (9) to (3);
		\draw [bend right=15,line width=1pt] (3) to (8);
        \draw [bend right=15,line width=8pt,opacity=0.2] (3) to (8);
		\draw [bend right=20,line width=1pt,dashed] (8) to (2);
		\draw [bend right=15,line width=1pt] (2) to (7);
        \draw [bend right=15,line width=8pt,opacity=0.2] (2) to (7);
		\draw [bend right=15,line width=1pt,dashed] (7) to (6);
		\draw [bend right=15,line width=1pt] (6) to (5);
        \draw [bend right=15,line width=8pt,opacity=0.2] (6) to (5);
		\draw [line width=1pt](4) to (10);
		\draw [line width=1pt](5) to (11);
	\end{pgfonlayer}
\end{tikzpicture}

    \caption{Representation of stretches containing the vertices of $B$}
    \label{fig:cycle_strech}
\end{figure}

From (2) and (3) of \Cref{observations} it follows that for each $j\geq 2$, 
\begin{equation}\label{nonadjacency of x_1}
    (t_{j}-1)\nsim x, \text{ and } (t_{j}-2)\nsim x.
\end{equation}
We further divide this case into two parts.

Let $y\sim (m-2)$. Again, (2) and (3) of \Cref{observations} imply that for each $j\geq 1$, 
\begin{equation}\label{nonadjacency of x_2}
    h_{j}+1\nsim x \text{ and } h_{j}+2\nsim x
\end{equation}
Also note that $y\nsim (m-3)$ as any two members of $B$ cannot be consecutive. Consider the strip 
\begin{equation*}
    S=\{1,2,\ldots, m-4\}
\end{equation*}
Recall that $q\geq 2$. Hence, \eqref{nonadjacency of x_1} and \eqref{nonadjacency of x_2} together implies that there exists at least one $j$ such that
\begin{align}
 (t_{j}-1)\nsim x, \text{ and } (t_{j}-2)\nsim x\\
    (h_{j}+1)\nsim x \text{ and } (h_{j}+2)\nsim x.
\end{align}
Thus, we need at least $2(m/2-2)-1+2=m-3$ vertices on $S$ to adjust the vertices of $(A\sqcup B)\setminus \{m-2,m-1\}$. However, $|S|=m-4$, which is a contradiction.

So let $y\nsim (m-2)$. In this case, we first show that there exists at most one $s$ such that $x\sim (h_{s}+1)$. If not, then let there exist two such $s$ and $s'$. This contradicts (2) of \eqref{observations}. So we consider that there exists no such $j$ and let the strip be
\begin{equation*}
    S=\{1,2,\ldots, m-3\}.
\end{equation*}
\eqref{nonadjacency of x_1} ensures that there exists at least one $j$ such that $t_{j}-1$ and $t_{j}-2$ are nonneighbor of $x$. Consequently, we need at least $2(m/2-1)-1+2=m-1$ many vertices on $S$ to adjust $A\sqcup B\setminus \{m-1\}$. This is not possible as $|S|=m-3$. Thus, we consider that there exists exactly one $s$ such that $x\sim (h_{s}+1)$. It follows from (2) and (3) of \Cref{observations} that $x\nsim (h_{j}+1)$ and $x\nsim (h_{j}+2)$ for all $j\geq 1$ and $j\neq s$. Since $q\geq 2$, there exists at least one such $j$. Here also, we consider the strip 
\begin{equation*}
    S=\{1,2,\ldots, m-3\}.
\end{equation*}
We need to plot $(m/2-1)$ vertices of $(A\sqcup B)\setminus \{m-1\}$ on $S$. However, this is not possible as the number of vertices on the required stretch is at least
\begin{equation*}
    2\left(\frac{m}{2}-1\right)-1+2=m-1.
\end{equation*}
This concludes that $l$ cannot be $(m-1)$. By using similar arguments, it can be shown that $l\neq (m-2)$. This completes the proof.
\end{proof}
We are now ready to prove \Cref{main theorem_2}.
\setcounter{theorem}{\getrefnumber{main theorem_2}-1}
\begin{theorem}
    If $m\geq 2n+2$ and $m\geq 3$, then $R(K_{2,n},C_{\{m,m+1\}})=m+1$.
\end{theorem}
\begin{proof}
We fix $m$ and $n$ such that $m\geq 2n+1$. To obtain the lower bound, we consider the graph $K_{1,m-1}$, which is $K_{2,n}$-free. The complement graph $\overline{G}$ does not contain a $C_{m}$ or $C_{m+1}$. This proves that $R(K_{2,n},C_{\{m,m+1\}})>m$. 

    For the upper bound, we consider a graph $G$ on $m+1$ vertices such that $G$ is $K_{2,n}$-free. It is enough to show that $\overline{G}$ either contains a $C_{m}$ or $C_{m+1}$. So, if possible, let $\overline{G}$ does not contain a $C_{m}$. This implies that $\overline{G}$ is $2$-connected by \Cref{2-connected lemma}. Since $G$ is $K_{2,n}$-free, for all $u,v\in V(\overline{G})$, we have
    \begin{equation*}
        |(N_{\overline{G}}(u)\cup N_{\overline{G}}(v))\setminus\{u,v\}|\geq (m-1)-(n-1)\geq \frac{m}{2}.
    \end{equation*}
     Hence, using the fact that $\overline{G}$ is $2$-connected, $ \overline{G}\nsupseteq C_{m}$ and applying \Cref{Hamiltonian lemma} on $\overline{G}$, we conclude that $\overline{G}$ is Hamiltonian. This completes the proof.
\end{proof}

\begin{theorem}
    If $m\geq 3n+4$, $R(K_{2,n},C_{m})=m+1$.
\end{theorem}
\begin{proof}
    The lower bound follows from Burr's inequality. Therefore, it is enough to prove the upper bound. For that, let us consider a graph on $m+1$ vertices such that $ G\nsupseteq K_{2,n}$ and $ \overline{G}\nsupseteq C_{m}$. From \Cref{main theorem_2}, it follows that $\overline{G}$ is Hamiltonian (and hence $2$-connected), that is, $\overline{G}\supseteq C_{m+1}$. It is an easy exercise to show that $\overline{G}$ is nonbipartite. We now divide the proof into two cases depending on the minimum degree of $\overline{G}$.

    Let $\delta(\overline{G})\geq (m+1)/4+250$. Since $\overline{G}$ is nonbipartite and $2$-connected, we can apply \Cref{Thm:Weakpan-Brandt_nby4_250_1} on $\overline{G}$ to conclude that $\overline{G}$ is weakly pancyclic and consequently $ \overline{G}\supseteq C_{m} $. This is a contradiction.

    So let $\delta(\overline{G})< (m+1)/4+250$. This implies that there exists a vertex $v\in V(G)$ such that
    \begin{equation*}
        |N_G(v)|>(m+1)-1-\frac{m+1}{4}-250=\frac{3(m+1)}{4}-251.
    \end{equation*}
    Since $G$ is $K_{2,n}$-free, for all $u\in V(G)$ and $u\neq v$,
    \begin{equation*}\label{eqn:min degree of G-v}
        |N_{\overline{G}}(u)\cap N_G(v)|>\left[\frac{3(m+1)}{4}-251\right]-1-(n-1)>\frac{m+2}{3}.
    \end{equation*}
    Therefore, we have $\delta(\overline{G}\setminus \{v\})> (m+2)/3$. Since $m\geq 3n+4$, we have $(m+2)/3\geq n+2$. As a consequence $\alpha(\overline{G}\setminus \{v\})\leq (m+2)/3$, otherwise we have a $K_{n+2}$ and therefore a $K_{2,n}$ in $G$. Using this fact and \Cref{eqn:min degree of G-v}, we get
    \begin{equation*}
        \delta(\overline{G}\setminus \{v\})\geq \max \bigg\{\frac{m+2}{3},\alpha(\overline{G}\setminus \{v\})\bigg\}
    \end{equation*}
    Now we divide the proof into two cases depending on the connectivity of $\overline{G}\setminus \{v\}$. Recall that $\overline{G}\setminus \{v\}$ is a graph on $m$ vertices such that $(G\setminus \{v\})\nsupseteq K_{2,n}$ and $(\overline{G}\setminus \{v\})\nsupseteq C_{m}$. Let $\kappa(\overline{G}\setminus \{v\})=0$ and $A$ be the smallest component of $\overline{G}\setminus \{v\}$. Hence,
    \begin{equation*}
        |A|\geq \delta(\overline{G}\setminus \{v\})>\frac{m+2}{3}\geq 2.
    \end{equation*}
    Also $|V(\overline{G})\setminus\{v\}|- |A|\geq m/2>n$. Thus we get a $K_{2,n}$ in $G$ by choosing any two vertices from $A$ and $n$ vertices from $V(\overline{G})\setminus(\{v\}\sqcup A)$.
    
    We obtain a similar contradiction if $\kappa(\overline{G}\setminus \{v\})=1$. As a consequence $\overline{G}\setminus \{v\}$ is $2$-connected. By using \Cref{Thm:Nash-william}, we conclude that $(\overline{G}\setminus \{v\})$ is Hamiltonian and hence $\overline{G}\supseteq C_{m}$, which is a contradiction. This completes the proof.
\end{proof}

\section{Constructions}
In this section, we construct graphs for  different sets of parameters, which gives a lower bound for $R(K_{2,n},C_{2m})$.

\setcounter{theorem}{\getrefnumber{Hamiltonian lemma}}

\begin{lemma}\label{lem:construction1}
    $K_{2,n}$ is not $C_{2m}$-good, that is,
    $$R(K_{2,n},C_{2m})>n+m+1,$$
    for any integer $n\in [m+2, \infty)\setminus\bigcup_{q=2}^{\infty}\{q(m+1)-2,q(m+1)-1,q(m+1)\}$.
\end{lemma}
\begin{proof}
    For $p\geq1$, consider the graph ${G}$ on $(p+1)m+t+1$ vertices defined as follows  
    $$\overline{G}=K_1\vee \left\{K_{m+t-p}\cup\left(\bigcup_{i=1}^pK_{m+1} \right)\right\},$$
    where $p+1\leq t< m+p-1$. Note that the size of the largest cycle in $\overline{G}$ is $m+t-p+1<2m$, therefore $\overline{G}\nsupseteq C_{2m}$. Further, note that $G$ is the complete multipartite graph with $p$ parts of size $m+1$ and one part of size $m+t-p$ together with an isolated vertex (for $p=2$, see \Cref{fig:placeholder}). Therefore, we have  $\max\{N_G(u)\cap N_G(v): u,v\in V(G)\}=pm+t-1$, which implies $G\nsupseteq K_{2,pm+t}$.
\end{proof}

\begin{figure}[!ht]
    \centering
    \begin{tikzpicture}[x=0.8cm,y=0.8cm]
	\begin{pgfonlayer}{nodelayer}
		\node [style=doty,scale=0.8] (0) at (-3.25, 0.5) {};
		\node [style=none] (1) at (-1.5, 1.25) {};
		\node [style=none] (2) at (0, 1.25) {};
		\node [style=none] (3) at (-0.75, 2) {};
		\node [style=none] (4) at (-0.75, 0.5) {};
		\node [style=none] (5) at (-4.5, -1.75) {};
		\node [style=none] (6) at (-3, -1.75) {};
		\node [style=none] (7) at (-3.75, -1) {};
		\node [style=none] (8) at (-3.75, -2.5) {};
		\node [style=none] (9) at (-5.75, 2.25) {};
		\node [style=none] (10) at (-4.25, 2.25) {};
		\node [style=none] (11) at (-5, 3) {};
		\node [style=none] (12) at (-5, 1.5) {};
		\node [style=doty,scale=0.8] (13) at (5.5, 0.5) {};
		\node [style=none] (14) at (7.25, 1.25) {};
		\node [style=none] (15) at (8.75, 1.25) {};
		\node [style=none] (16) at (8, 2) {};
		\node [style=none] (17) at (8, 0.5) {};
		\node [style=none] (18) at (4.25, -1.75) {};
		\node [style=none] (19) at (5.75, -1.75) {};
		\node [style=none] (20) at (5, -1) {};
		\node [style=none] (21) at (5, -2.5) {};
		\node [style=none] (22) at (3, 2.25) {};
		\node [style=none] (23) at (4.5, 2.25) {};
		\node [style=none] (24) at (3.75, 3) {};
		\node [style=none] (25) at (3.75, 1.5) {};
		\node [style=none] (26) at (-4.75, 2.75) {};
		\node [style=none] (27) at (-4.5, 2.75) {};
		\node [style=none] (28) at (-4.5, 2.5) {};
		\node [style=none] (29) at (-4.75, 2.25) {};
		\node [style=none] (30) at (-4.5, 2) {};
		\node [style=none] (31) at (-4.75, 2) {};
		\node [style=none] (32) at (-5, 2) {};
		\node [style=none] (33) at (-5.25, 1.75) {};
		\node [style=none] (34) at (-0.75, 1.75) {};
		\node [style=none] (35) at (-1, 1.5) {};
		\node [style=none] (36) at (-1.25, 1.25) {};
		\node [style=none] (37) at (-1, 1.25) {};
		\node [style=none] (38) at (-1.25, 1) {};
		\node [style=none] (39) at (-1, 1) {};
		\node [style=none] (40) at (-1, 0.75) {};
		\node [style=none] (41) at (-0.75, 0.75) {};
		\node [style=none] (42) at (-4.25, -1.75) {};
		\node [style=none] (43) at (-4.25, -1.5) {};
		\node [style=none] (44) at (-4.25, -1.5) {};
		\node [style=none] (45) at (-4.25, -1.5) {};
		\node [style=none] (46) at (-4, -1.25) {};
		\node [style=none] (47) at (-3.75, -1.25) {};
		\node [style=none] (48) at (-3.75, -1.25) {};
		\node [style=none] (49) at (-3.25, -1.5) {};
		\node [style=none] (50) at (-3.25, -1.75) {};
		\node [style=none] (51) at (-3.5, -1.5) {};
		\node [style=none] (52) at (4.25, 2.5) {};
		\node [style=none] (53) at (4.25, 2.25) {};
		\node [style=none] (54) at (4, 2.25) {};
		\node [style=none] (55) at (3.75, 2) {};
		\node [style=none] (56) at (7.5, 1.75) {};
		\node [style=none] (57) at (7.75, 1.5) {};
		\node [style=none] (58) at (8, 1.25) {};
		\node [style=none] (59) at (8, 1) {};
		\node [style=none] (60) at (4.5, -1.5) {};
		\node [style=none] (61) at (4.75, -1.5) {};
		\node [style=none] (62) at (5, -1.5) {};
		\node [style=none] (63) at (5.25, -1.5) {};
		\node [style=none] (64) at (-5, 3.75) {$\overline{K}_{m+t-2}$};
		\node [style=none] (65) at (-0.5, 2.5) {$\overline{K}_{m+1}$};
		\node [style=none] (66) at (-2, -2) {$\overline{K}_{m+1}$};
		\node [style=none] (67) at (3.75, 3.5) {$K_{m+t-2}$};
		\node [style=none] (68) at (8.0, 2.5) {$K_{m+1}$};
		\node [style=none] (69) at (6.65, -2) {$K_{m+1}$};

        \node [style=none] (70) at (-4, -3) {$G$};

        \node [style=none] (71) at (5, -3) {$\overline{G}$};
	\end{pgfonlayer}
	\begin{pgfonlayer}{edgelayer}
		\draw [bend left=45] (3.center) to (2.center);
		\draw [bend left=45] (2.center) to (4.center);
		\draw [bend left=45] (4.center) to (1.center);
		\draw [bend left=45] (1.center) to (3.center);
		\draw [bend left=45] (7.center) to (6.center);
		\draw [bend left=45] (6.center) to (8.center);
		\draw [bend left=45] (8.center) to (5.center);
		\draw [bend left=45] (5.center) to (7.center);
		\draw [bend left=45] (11.center) to (10.center);
		\draw [bend left=45] (10.center) to (12.center);
		\draw [bend left=45] (12.center) to (9.center);
		\draw [bend left=45] (9.center) to (11.center);
		\draw [bend left=45] (16.center) to (15.center);
		\draw [bend left=45] (15.center) to (17.center);
		\draw [bend left=45] (17.center) to (14.center);
		\draw [bend left=45] (14.center) to (16.center);
		\draw [bend left=45] (20.center) to (19.center);
		\draw [bend left=45] (19.center) to (21.center);
		\draw [bend left=45] (21.center) to (18.center);
		\draw [bend left=45] (18.center) to (20.center);
		\draw [bend left=45] (24.center) to (23.center);
		\draw [bend left=45] (23.center) to (25.center);
		\draw [bend left=45] (25.center) to (22.center);
		\draw [bend left=45] (22.center) to (24.center);
		\draw (26.center) to (34.center);
		\draw (30.center) to (36.center);
		\draw (28.center) to (37.center);
		\draw (29.center) to (35.center);
		\draw (38.center) to (51.center);
		\draw (40.center) to (50.center);
		\draw (39.center) to (48.center);
		\draw (41.center) to (49.center);
		\draw (31.center) to (47.center);
		\draw (12.center) to (46.center);
		\draw (33.center) to (45.center);
		\draw (32.center) to (42.center);
		\draw (56.center) to (13);
		\draw (57.center) to (13);
		\draw (58.center) to (13);
		\draw (59.center) to (13);
		\draw (13) to (63.center);
		\draw (13) to (62.center);
		\draw (13) to (61.center);
		\draw (13) to (60.center);
		\draw (52.center) to (13);
		\draw (53.center) to (13);
		\draw (54.center) to (13);
		\draw (55.center) to (13);
	\end{pgfonlayer}
\end{tikzpicture}
    \caption{Graph $G$ on $3m+t+1$ vertices, $G\nsupseteq K_{2,2m+t}$ and $\overline{G}\nsupseteq C_{2m}$, $3\leq t< m+1$.}
    \label{fig:placeholder}
\end{figure}
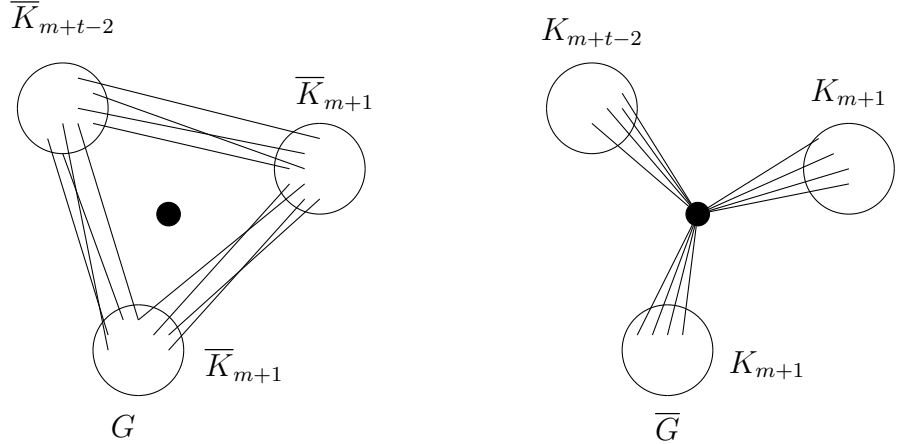

\begin{lemma}\label{lem:construction2}
    $K_{2,n}$ is not $C_{2m}$-good, that is,
    $$R(K_{2,n},C_{2m})>n+m+1,$$
    for any integer $n\in \bigcup_{q=2}^{\infty}\{q(m+1)-2,q(m+1)-1,q(m+1)\}$.
\end{lemma}
\begin{proof}
For $q\geq 2$ and $t\in\{0,1,2\}$, take $n=q(m+1)-t$ and consider the graph $G$ on $n+m+1$ vertices such that 
    $$\overline{G}=K_1\vee \left\{K_{2m-t-2}\cup K_{m+4}\cup\left(\bigcup_{i=1}^{q-2}K_{m+1} \right)\right\}.$$
    Clearly, $\overline{G}\nsupseteq C_{2m}$. Further, note that $G$ is the complete multipartite graph with $q-2$ parts of size $m+1$, one part of size $m+4$ and one part of size $2m-t-2$. Therefore, we have  $\max\{|N_G(u)\cap N_G(v)|: u,v\in V(G)\}=n-1$, which implies $G\nsupseteq K_{2,n}$.
\end{proof}

The proof of \Cref{thm:RamseyBadness in n>m+1} follows from the \Cref{lem:construction1,lem:construction2}.

\begin{acknowledgement}
This research work has no associated data. The work of the author Abisek Dewan is supported by University Grants Commission, India (Beneficiary Code/Flag: BWBDA00147662 U). The author Sayan Gupta thanks NISER Bhubaneswar and Homi Bhabha National Institute (HBNI), Mumbai for funding his PhD fellowship. The work of the author Rajiv Mishra is supported by Council of Scientific $\And$ Industrial Research, India(File number: 09/921(0347)/2021-EMR-I).  
\end{acknowledgement}

\bibliographystyle{siam}
	\bibliography{goodbib}
\end{document}